\documentclass[11pt]{amsart}
\usepackage{microtype}
\usepackage{mathtools,amssymb}
\usepackage{graphicx}
\usepackage{comment,overpic,epsfig}
\usepackage[usenames,dvipsnames]{color}
\usepackage[hyphens]{url}
\usepackage[pagebackref,linktocpage=true,colorlinks=true,linkcolor=RoyalBlue,citecolor=BrickRed,urlcolor=RoyalBlue]{hyperref}
\usepackage[msc-links,abbrev]{amsrefs}
\usepackage{mleftright}
\usepackage[T1]{fontenc}

\newcommand{\R}{\mathbf{R}}
\newcommand{\ol}{\overline}

\newcommand{\C}{\mathcal{C}}
\renewcommand{\tilde}{\widetilde}
\renewcommand{\S}{\mathbf{S}}
\renewcommand{\phi}{\varphi}
\renewcommand{\epsilon}{\varepsilon}
\renewcommand{\leq}{\leqslant}
\renewcommand{\geq}{\geqslant}

\DeclareMathOperator{\conv}{conv}

\DeclareMathOperator{\ave}{ave}
\DeclareMathOperator{\cm}{cm}
\DeclareMathOperator{\Emb}{Emb}
\DeclareMathOperator{\Imm}{Imm}
\DeclareMathOperator{\length}{length}

\DeclareMathOperator{\mass}{m}

\theoremstyle{plain}
\newtheorem{theorem}{Theorem}[section]
\newtheorem{proposition}[theorem]{Proposition}
\newtheorem{lemma}[theorem]{Lemma}
\newtheorem{corollary}[theorem]{Corollary}
\theoremstyle{remark}

\theoremstyle{definition}

\begin{document}

\vspace*{-0.5in}
	
\title[Curves and knots with prescribed curvature]{$h$-Principles for smooth curves and knots\\with prescribed curvature}

\author{Mohammad Ghomi}
\address{School of Mathematics, Georgia Institute of Technology, Atlanta, Georgia 30332}
\email{ghomi@math.gatech.edu}
\urladdr{https://ghomi.math.gatech.edu}

\author{Matteo Raffaelli}
\address{School of Mathematics, Georgia Institute of Technology, Atlanta, Georgia 30332}
\email{raffaelli@math.gatech.edu}
\urladdr{https://matteoraffaelli.com}

\date{Last revised on \today}
\thanks{The first-named author was supported by NSF grant DMS-2202337}

\begin{abstract}
We show that smooth curves with prescribed curvature satisfy a $\C^1$-dense $h$-principle in the space of immersed curves in Euclidean space. More precisely, every $\C^{\alpha \geq 2}$ curve with nonvanishing curvature in  $\R^{n\geq 3}$ can be $\C^1$-approximated by $\C^\alpha$ curves of any larger curvature, prescribed as a function of arclength. It follows that there exist $\C^\infty$ knots of prescribed curvature in every isotopy class of closed curves embedded in $\R^3$.
\end{abstract}

\maketitle

\section{Introduction}

Wasem \cite{wasem2016} showed that every $\C^{2}$ curve in Euclidean space $\R^{n\geq 3}$ may be $\C^1$-approximated by curves of  larger prescribed curvature. The proof, which is based on the work of Nash and Kuiper~\cite{nash1954, kuiper1955, conti2012}, requires the target curvature function to be $\C^\infty$, but produces only $\C^{2}$ curves. Adopting a more geometric approach, we refine Wasem’s result to obtain optimal regularity. We also achieve tighter control over the speed and position of the approximating curves.

To state our main result, let $\Gamma$ be the interval $[a,b]\subset\R$ or a topological circle $\R/((b-a)\textbf{Z})$, and $\C^\alpha(\Gamma,\R^n)$ be the space of $\C^\alpha$ maps $f\colon \Gamma\to\R^n$ with the topology induced by the $\C^\alpha$-norm 
$\lvert \,\cdot\, \rvert_\alpha$. The space of (immersed) \textit{curves} $\Imm^{\alpha\geq 1}(\Gamma,\R^n)\subset \C^\alpha(\Gamma,\R^n)$ consists of maps with nonvanishing derivative. The \textit{curvature} of $f\in\Imm^{2}(\Gamma,\R^n)$ is given by $\kappa\coloneqq \lvert T' \rvert / \lvert f'\rvert$, where $T\coloneqq f'/\lvert f'\rvert$ is the \textit{tantrix} of $f$.

\begin{theorem}\label{TH:main}
Let $f\in\Imm^{\alpha\geq 2}(\Gamma,\R^{n\geq 3})$ be a curve with curvature $\kappa>0$, and $\tilde\kappa\colon\Gamma\to\R$ be a $\C^{\alpha-2}$ function with $\tilde\kappa>\kappa$. Then, for any $\epsilon>0$, there exists a curve $\tilde f\in\Imm^\alpha(\Gamma,\R^n)$ with curvature $\tilde\kappa$ and $\lvert \tilde f-f \rvert_1\leq\epsilon$. If $f$ has unit speed, then so does $\tilde f$. Furthermore, $\tilde f$ can be made tangent to $f$ at any finite set of points prescribed along $\Gamma$.
\end{theorem}

The last two properties of $\tilde f$ are additional features of our method, not granted by the Nash--Kuiper approach in \cite{wasem2016}. If $\epsilon$ in Theorem~\ref{TH:main} is sufficiently small,  $h_t\coloneqq  (1-t)f+t \tilde f$, $t\in[0,1]$, is a homotopy in $\Imm^{\alpha}(\Gamma,\R^n)$. Thus, in the terminology of Gromov~\cite{gromov:PDR} or Eliashberg~\cite{eliashberg2002}, the above result establishes a $\C^1$-dense $h$-principle for smooth curves of prescribed curvature. Similar $h$-principles for curves of constant curvature or constant torsion were obtained in \cite{ghomi2007-h, ghomi-raffaelli2024b,ghomi-raffaelli2024}. Theorem~\ref{TH:main} has the following quick application. Let $\Emb^\alpha(\Gamma, \R^3)\subset \Imm^\alpha(\Gamma, \R^3)$ be the space of injective curves, which are called \textit{knots} when $\Gamma$ is a circle. 

\begin{corollary}\label{COR:main}
Let $f \in\Emb^{\alpha\geq 2}(\Gamma, \R^3)$ be a knot, and $\tilde\kappa \colon\Gamma \to \R$ be a positive $\C^{\alpha-2}$ function. Then $f$ is isotopic in $\Emb^\alpha(\Gamma, \R^3)$ to a constant-speed knot of curvature $\tilde\kappa$.
\end{corollary}
\begin{proof}
After an isotopy, we may assume that $f$ has constant speed. Choose $\lambda$ so large that the curvature of $f$ is smaller than $\lambda\tilde\kappa$. Then, by Theorem \ref{TH:main}, $f$ is $\C^1$-close and therefore isotopic to a unit-speed curve $\tilde f$ with curvature $\lambda \tilde\kappa$. Finally, $\tilde f$ is isotopic to $\lambda\tilde f$ which has curvature $\tilde\kappa$.
\end{proof}

The above corollary generalizes \cite[Cor.~1]{wasem2016}, where $\alpha=2$ and the speed was not constant. See also \cite{ghomi2007-h,ghomi-raffaelli2024,ghomi-raffaelli2024b} where this result had been obtained for constant $\tilde\kappa$. 
The existence of $\C^2$ knots of constant curvature was first established by McAtee~\cite{mcatee2007}. 

As in \cite{ghomi2007-h, ghomi-raffaelli2024, ghomi-raffaelli2024b}, the proof of Theorem~\ref{TH:main} is based on a variant of convex integration~\cite{gromov:PDR, spring1998, eliashberg2002, geiges2003}. It hinges on the fact that when $\lvert f'\rvert =1$,  $\kappa=|T'|$. After reparametrizing $f$ with unit speed, we deform $T$ to a (longer) spherical curve $\tilde T$ with speed $\tilde\kappa$, which we then integrate to obtain $\tilde f$. For $\tilde f$ to be $\C^1$-close to $f$, $\tilde T$ should be $\C^0$-close to $T$. Moreover, for $\tilde f$ to be tangent to $f$ at the prescribed points $p_i\in\Gamma$, $\tilde T$ should have the same integral as $T$ on every interval between $p_i$. These requirements will be met by combining basic convex geometry (Lemma~\ref{LM:Kalman}) with mapping degree theory (Lemma~\ref{LM:top}). The latter is the novel feature of this work, which makes the arguments significantly shorter than those in \cite{ghomi2007-h}.

\section{Preliminaries}

We begin by recording some basic facts from \cite{ghomi2007-h} which are needed here.

\subsection{Unit speed reparametrization}\label{subsec:reparam}
It suffices to establish Theorem \ref{TH:main} for unit-speed curves. To see this, note that for any $f \in\C^1(\Gamma,\R^n)$ and $\phi \in \C^1(\Gamma,\Gamma)$, we have \cite[Lem.~5.1]{ghomi2007-h}
\begin{equation*}
\lvert f\circ\phi \rvert_1 \leq \lvert f \rvert_1(1+ \lvert \phi\rvert_1).
\end{equation*}
Suppose that Theorem~\ref{TH:main} holds for unit-speed curves. Given $f\in\Imm^\alpha(\Gamma,\R^n)$,  
choose $\lambda > 0$ such that $\length(\lambda f)=b-a$. Then there exists a $\C^\alpha$ diffeomorphism $\phi \colon \Gamma \to \Gamma$ such that $\lambda f \circ \phi$ has unit speed. Let $\tilde{\lambda f \circ \phi}$ be the unit-speed curve with curvature $\lambda\,\tilde \kappa\circ\phi$ such that
\begin{equation*}
\bigl\lvert \tilde{\lambda f \circ \phi} - \lambda f \circ \phi \bigr\rvert_1 \leq \epsilon \lambda/\mleft(1+\rvert \phi^{-1}\rvert_1\mright),
\end{equation*}
and $\tilde{\lambda f \circ \phi}$ is tangent to $\lambda f \circ \phi$ at $\phi^{-1}(p_i)$, where $p_i$ are the prescribed points.
Set $\tilde f\coloneqq \tilde{\lambda f \circ \phi}\circ\phi^{-1}/\lambda$. Then $\tilde f$ has curvature $\tilde \kappa$, and is tangent to $f$ at $p_i$. Furthermore, 
\begin{equation*}
\bigl\lvert \tilde f - f \bigr\rvert_1 
= 
\bigl\lvert  \bigl( \tilde{\lambda f \circ \phi} - \lambda f\circ\phi\bigr) \circ\phi^{-1} \bigr\rvert_1/\lambda
\leq 
 \bigl \lvert \tilde{\lambda f \circ \phi} - \lambda f\circ\phi\bigr\rvert_1 \bigl(1+ \lvert \phi^{-1}\rvert_1\bigr)/\lambda 
 \leq
 \epsilon,
\end{equation*}
as desired.

\subsection{Average and center of mass}
Let $I\coloneqq [a,b]$ and $\lvert I\rvert \coloneqq b-a$. For any curve $f\in  \Imm^{\alpha\geq 1}(I,\R^n)$, set $\length(f)\coloneqq \int_I \lvert f'\rvert\,du$. The  \textit{average} and \textit{center of mass} of $f$ are defined as 
\begin{equation*}
\ave(f)\coloneqq \frac{1}{\lvert I \rvert}\int_I f\,dt,
\quad\quad \text{and} \quad\quad 
\cm (f)\coloneqq \frac{1}{\length(f)} \int_I f\lvert f'\rvert\,dt.
\end{equation*}
In particular, $\cm(f)=\ave(f)$ when $f$ has constant speed. Moreover, if $g\colon [c,d]\to\R^n$ is a \emph{reparametrization} of $f$, i.e., there exists a diffeomorphism $\phi \colon [c,d]\to[a,b]$ such that $g = f\circ\phi$, then $\cm(f) = \cm(g)$. 
For any positive (density) function $\rho \in \C^0(I,\R)$, we define the corresponding \emph{mass} and \emph{center of mass} of $f$ as 
\begin{equation*}
\mass(f,\rho)\coloneqq \int_I \rho \lvert f'\rvert \, dt, \quad\quad\text{and}\quad\quad  \cm(f,\rho) \coloneqq \frac{1}{\mass(f,\rho)}\int_I f\rho \lvert f'\rvert\,dt.
\end{equation*}
So $\cm(f,1/\lvert f' \rvert) = \ave(f)$. More generally, letting $\phi\colon[0,\mass(f,\rho)]\to I$ be the  inverse of the (mass) function $t\mapsto \int_a^t \rho \lvert f'\rvert \, du$, we obtain
\begin{equation}\label{eq:cm-ave2}
\cm(f,\rho) = \ave(f\circ\phi),
\end{equation}
see \cite[Lem.~2.2]{ghomi2007-h}.

\section{Reduction to Spherical Curves}
Here we show that Theorem~\ref{TH:main} follows from a geometric result (Proposition \ref{PROP:spherical}) for spherical curves.
First we reduce Theorem~\ref{TH:main} to the following local problem.

\begin{proposition}\label{PROP:local}
Let $f\in\Imm^{\alpha\geq 2}(I,\R^n)$ be a unit-speed curve with curvature $\kappa>0$, and $V$ be an open neighborhood of $T(I)$ in $\S^{n-1}$. Then for any $\C^{\alpha-2}$ function $\tilde\kappa\colon I\to\R$ with $\tilde\kappa>\kappa$, there exists a unit-speed curve $\tilde f\in\Imm^\alpha(I,\R^n)$ with curvature $\tilde\kappa$ such that $\tilde T(I)\subset V$, $f = \tilde f$ on $\partial I$, and $T(U) =\tilde T(\tilde U)$ for some open neighborhoods $U$, $\tilde U$ of $\partial I$ in $I$.
\end{proposition}

To see that the above proposition implies Theorem~\ref{TH:main} recall that,  as discussed in Section \ref{subsec:reparam}, we may assume that $f$ in Theorem \ref{TH:main} has unit speed. So $T=f'$. Let $I_i$ be a partition of $\Gamma$ into intervals such that $\partial I_i$ include all the prescribed points, and set $f_i\coloneqq f\rvert_{I_i}$. Choose $I_i$ so small that  $\lvert I_i\rvert \leq 1$ and $T(I_i)$ lies in the interior $V_i$ of a ball of radius $\epsilon/2$ in $\S^{n-1}$. Applying Proposition~\ref{PROP:local} to $f_i$, we obtain a $\C^\alpha$ curve $\tilde f_i$ with unit speed and curvature $\tilde\kappa_i\coloneqq \tilde\kappa\rvert_{I_i}$ such that $\tilde T_i(I_i)\subset V_i$, $f_i = \tilde f_i$ on $\partial I_i$, and $T_i(U_i) =\tilde T_i(\tilde U_i)$ for some open neighborhoods $U_i$, $\tilde U_i$ of $\partial I_i$ in $I_i$. 

Define $\tilde f$ by $\tilde f\rvert_{I_i}\coloneqq \tilde f_i$. Then $\tilde f$ is $\C^0$, because $f_i = \tilde f_i$ on $\partial I_i$. Moreover $\tilde f$ is $\C^1$, because $\tilde T_i= \tilde f_i'$ and $\tilde T_i=T_i$ on $\partial I_i$. Hence $\tilde T \coloneqq \tilde f'$ is well-defined.  Note that $\tilde T$ is piecewise $\C^{\alpha-1}$. Furthermore, since $T_i(U_i)=\tilde T_i(\tilde U_i)$, there are open neighborhoods $W_i$ and $\tilde W_i$ of $\partial I_i$ in $\Gamma$ such that $\tilde T(\tilde W_i) = T(W_i)$. Hence $\tilde T$ is a reparametrization of $T$ with speed $\tilde \kappa$ near $\partial I_i$, and so is $\C^{\alpha-1}$. Consequently $\tilde f$ is $\C^\alpha$ and has curvature $\tilde \kappa$. Next note that for  $t\in I_i$, $\tilde f(t) = f(t_i) + \int_{t_i}^t \tilde T_i \,du$ and $f(t) = f(t_i) + \int_{t_i}^t  T_i \,du$, where $t_i$ is the initial point of $I_i$.
Furthermore, since $\tilde T(I_i) \subset V_i$, we have $\lvert \tilde T - T\rvert_0 \leq \epsilon$. Thus,
\begin{equation*}
\bigl\lvert \tilde f(t) - f(t) \bigr\rvert 
= 
\left\lvert \int_{t_i}^t (\tilde T_i - T_i)  \,du\right\rvert 
\leq
 \int_{t_i}^t \left\lvert\tilde T_i - T_i\right\rvert du
\leq
 \epsilon \lvert I_i\rvert 
 \leq 
 \epsilon.
\end{equation*}
So Proposition \ref{PROP:local} does indeed imply Theorem \ref{TH:main}.  Next we show that Proposition~\ref{PROP:local}  is a consequence of the following result for spherical curves. We say that a curve $f\in\Imm^\alpha(I,\R^n)$ is \emph{nonflat} if the convex hull  of $f(I)$ has interior points.

\begin{proposition}\label{PROP:spherical}
Let $T\in\Imm^{\alpha\geq 1}(I,\S^{n-1})$ be a nonflat curve. Then for any $\C^{\alpha-1}$ function $\tilde v\colon I\to\R$ with $\tilde v>\lvert T' \rvert$ and open neighborhood $V$ of $T(I)$, there exists a curve $\tilde T\in\Imm^\alpha(I,\S^{n-1})$ such that 
$\lvert \tilde T'\rvert=\tilde v$,
$\tilde T(I)\subset V$,
$T(U) =\tilde T(\tilde U)$ for some open neighborhoods $U$, $\tilde U$ of $\partial I$, and
$\ave(T)=\ave(\tilde T)$.
\end{proposition}

To see that Proposition~\ref{PROP:spherical} implies Proposition~\ref{PROP:local}, let $T$ be the tantrix of $f$ in Proposition~\ref{PROP:local}. After a small $\C^\alpha$ perturbation of $f$ on a compact set in the interior of $I$, we may assume that $T$ is nonflat. This is possible since by assumption $\kappa>0$. Hence $f$ does not trace a line segment, and so it can be perturbed without changing its length, which ensures that it remains a unit-speed curve. Furthermore, since the perturbation is $\C^2$-close to the original curve, we will still have $\tilde \kappa>\kappa$. Now let $\widetilde T$ be the $\C^{\alpha-1}$ curve obtained by applying Proposition~\ref{PROP:spherical} to $T$ with $\tilde v\coloneqq \tilde\kappa$. Then $\tilde f(t)\coloneqq f(a) + \int_a^t \tilde T\,du$ is a $\C^\alpha$ curve with $\tilde f' =\tilde T$. So $\tilde f$ has unit speed and curvature $\tilde \kappa$. Finally, the assumption that $\ave(T)=\ave(\tilde T)$ ensures that $f = \tilde f$ on $\partial I$, which completes the argument.

\section{Proof of Theorem~\ref{TH:main}}
It remains to establish Proposition~\ref{PROP:spherical}, which completes the proof of Theorem~\ref{TH:main} as discussed above.
Our proof is based on the following topological lemma, which is established quickly via basic degree theory.

\begin{lemma}[\cite{ghomi-raffaelli2024b}]\label{LM:top}
Let $B\subset \R^n$ be a ball of radius $R$ centered at $x_0$, and $F\colon B\to\R^n$ be a continuous map. If $\lvert F(x) - x\rvert < R$ for all $x\in\partial B$, then $x_0\in F(B)$.
\end{lemma}

We also need the following version of a classical result of  Kalman~\cite{kalman1961}, who constructed continuous barycentric coordinates in convex polytopes. 

\begin{lemma}\label{LM:Kalman}
Let $p_1,\dots,p_k\in\R^n$ be a collection of points whose convex hull has nonempty interior, and $B$ be a ball centered at $x_0\coloneqq \sum_{i=1}^k p_i/k$.  If $B$ is sufficiently small, then there are positive $\C^\infty$ functions $c_i \colon B \to \R$ such that 
$$
\sum_{i=1}^k c_i(x) =1,\quad\quad\sum_{i=1}^k c_i(x)p_i = x,\quad\quad c_{i}(x_0) = \frac{1}{k}.
$$
\end{lemma}

\begin{proof}
We may assume that $x_0=0$ after a translation. Then there are $n$ linearly independent vectors among $p_i$, say $p_1,\dotsc, p_n$. So there are unique coefficients $a_1(x),\dotsc, a_n(x)$ such that $x = \sum_{i=1}^n a_i(x)p_i$ for each $x\in B$. The functions $a_i\colon B\to\R$ are $\C^\infty$, as they are linear.
Set $a_i =0$ for $i>n$, and let
\begin{equation*}
c_i(x)\coloneqq \frac{1}{k} \left(1-\sum_{i=1}^k a_i(x)\right)+ a_i(x).
\end{equation*}
Since $\sum_{i=1}^n a_i(x_0)p_i=\sum_{i=1}^k a_i(x_0)p_i=x_0$ and $p_1,\dotsc, p_n$ are linearly independent, $a_i(x_0)=0$. Thus $c_i(x_0)=1/k$. So $c_i>0$ on $B$ if $B$ is small. Furthermore, 
\begin{gather*}
\sum_{i=1}^k c_i(x) = 1-\sum_{i=1}^k a_i(x)+\sum_{i=1}^k a_i(x)=1,\\
\sum_{i=1}^k c_i(x)p_i = x_0 \biggl(1-\sum_{i=1}^k a_i(x) \biggr)+\sum_{i=1}^k a_i(x)p_i=x,
\end{gather*}
as desired.
\end{proof}

Now we are ready to prove the key result of this work:

\begin{proof}[Proof of Proposition~\textup{\ref{PROP:spherical}}]
Let $\psi\colon [0,1]\to I$ be given by $\psi(t)\coloneqq (b-a)t+a$. Then $\ave(T\circ\psi)=\ave(T)$. So, after replacing $T$ with $T\circ\psi$ and $v$ with $v\circ\psi$, we may assume that $I=[0,1]$.
Then
\begin{equation*}
x_0 \coloneqq \ave(T)=\int_I T.
\end{equation*}
 Since $T$ is nonflat, there exists a  ball $B$ of radius $R>0$ in the convex hull of $T(I)$ centered at $x_0$.
By Lemma~\ref{LM:top} it is enough to construct a continuous mapping 
\begin{equation*}
B \ni x \mapsto \tilde T_x \in \Imm^\alpha(I,\S^{n-1})
\end{equation*}
such that  $\tilde T_x$ satisfies the first three required properties in the statement of the proposition, and $\lvert \ave(\tilde T_x) - x\rvert < R$. 
To this end we first construct a continuous family of reparametrizations $\ol T_x \in \Imm^\alpha(I,\S^{n-1})$ of $T$ such that 
\begin{equation}\label{eq:olT}
\left\lvert \ave(\ol T_x) - x\right\rvert < R/2, \qquad\text{and}\qquad \ol T_{x_0} = T.
\end{equation}
 Then we construct $\tilde T_x$ by adding certain loops to $\ol T_x$ so that
$
\lvert \ave(\tilde T_x) - \ave(\ol T_x) \rvert < R/2,
$
which will complete the proof.

(\emph{Part I}) We need to find a continuous family of $\C^\alpha$ diffeomorphism $\phi_x\colon I\to I$ such that $\ol T_x:=T\circ\phi_x$ satisfies the conditions listed in \eqref{eq:olT}. To this end it suffices to construct a continuous family of positive $\C^{\alpha-1}$ functions $\rho_x\colon I\to\R$ such that 
\begin{equation}\label{eq:rho}
\lvert \cm(T,\rho_x)-x \rvert< R/2, \qquad \rho_{x_0}= 1/\lvert T'\rvert, \qquad \text{and} \qquad \int_I\rho_x|T'|dt=1.
\end{equation}
Then the inverse of the function $[0,1]\ni t\mapsto \int_0^t \rho_x \lvert T'\rvert \, du$ gives the desired $\phi_x$.
Indeed the last condition in \eqref{eq:rho} ensures that $\phi_x\colon I\to I$, and since $\rho_x$ is $\C^{\alpha-1}$ and positive, $\phi_x$ is a $\C^\alpha$ diffeomorphism by the inverse function theorem. Furthermore, the first condition in \eqref{eq:rho} yields the first condition in \eqref{eq:olT} by \eqref{eq:cm-ave2}. Finally,
the second condition in \eqref{eq:rho} guarantees that $\phi_{x_0}$ is the identity, which yields the second condition in \eqref{eq:olT}. 

To construct $\rho_x$ we rewrite \eqref{eq:rho} as 
\begin{equation}\label{eq:olrho}
\left\lvert  \int_I\ol\rho_xT\,dt - x\right\rvert < R/2,\quad\quad  \ol\rho_{x_0}= 1, \quad\quad\text{and}\quad\quad \int_I\ol\rho_{x}=1,
\end{equation}
where $\ol\rho_x:=\lvert T'\rvert \rho_x$. Hence it suffices to construct a continuous family of positive $\C^{\alpha-1}$ functions $\ol\rho_x\colon I\to \R$ satisfying \eqref{eq:olrho}. To this end let $I_i$ be a partition of $I$ into $k>n$ equal segments, and $\theta_i$ be a $\C^\infty$ partition of unity subordinate to $I_i$.  Set 
$$
p_i\coloneqq \ave(\theta_i T)=k \int_I \theta_i T\,dt.
$$
 Then
\begin{equation*}
x_0 =  \sum_i \int_I \theta_i T \, dt = \frac{1}{k} \sum_i p_i.
\end{equation*}
Assuming $k$ is sufficiently large, the convex hull of $p_i$, $\conv(\{p_i\})$, has interior points, since $T$ is nonflat. In particular $x_0$ lies in the interior of $\conv(\{p_i\})$. Fix $k$ and choose $R$ so small that $B$ lies in $\conv(\{p_i\})$.
Then, by Lemma~\ref{LM:Kalman}, there exist positive $\C^\infty$ functions $c_i \colon B\to \R$ with $\sum c_i =1$ such that $\sum c_i(x)p_i = x$ and $c_i(x_0) = 1/k$.
Set
\begin{equation*}
\ol \rho_x \coloneqq \lambda(x) \sum_i c_i(x)\theta_i, \quad \text{where}\quad \lambda(x)\coloneqq 1\big/\int_I \sum_i c_i(x)\theta_i\, dt.
\end{equation*}
Then $\ol\rho_{x_0}=1$, and $\int_I\ol\rho_{x}=1$. Furthermore, 
\begin{equation*}
\int_I\ol\rho_xT\, dt =\lambda(x) \sum_i c_i(x)\int_I \theta_i T\,dt =\frac{\lambda(x)}{k}\sum_i c_i(x)p_i=\frac{\lambda(x)}{k} x.
\end{equation*}
Since $\lambda(x_0)=k$, choosing $R$ sufficiently small we can make sure that $\lvert \lambda(x)/k - 1\rvert <1/2$. Then 
\begin{equation*}
\left\lvert \int_I\ol\rho_xT\,dt-x\right\rvert 
= 
\left\lvert \frac{\lambda(x)}{k}-1\right\rvert \lvert x\rvert 
< 
 \frac{R}{2},
\end{equation*}
as desired.

(\emph{Part II})
Partition $I$ into subintervals $I_i$ such that for all $x\in B$, $\ol T_x$ is injective on $I_i$ and $\length(\ol T_x|_{I_i})<R/2$. So $\ol T_x(I_i)$ lies in a ball $U_i^x$ of radius less than $R/4$ centered at the midpoint $q_i^x$ of $\ol T_x(I_i)$. Let $C_i^x\subset (V\cap U_i^x)$ be a family of $\C^\alpha$ loops, depending continuously on $q_i^x$, that coincide with $\ol T_x(I_i)$ near $q_i^x$.  For instance, we may construct $C_i^x$ by gluing a segment of $\ol T_x(I_i)$ to an arc of a circle centered at $q_i^x$ and rounding off the corners. Similarly we may construct a continuous family of $\C^\alpha$ loops $C_{i,\ell}^x$ nested inside $C_i^x$ with length $0<\ell\leq\length(C_i^x)$  that contain a neighborhood of $q_i^x$ in $\ol T_x(I_i)$; see Figure~\ref{fig:loop}. 
\begin{figure}[h]
\begin{overpic}[height=1.25in]{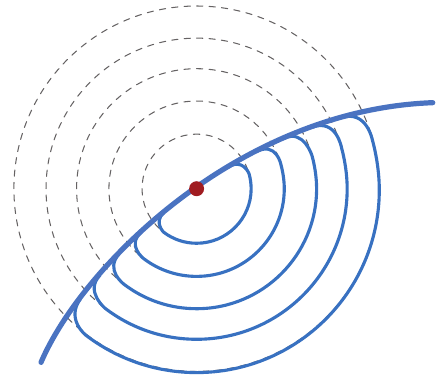}
\put(-18,-2){\small $\ol T_x(I_i)$}
\put(35,45){\small $q_i^x$}
\put(81,15){\small $C_{i,\ell}^x$}
\end{overpic}
\caption{}\label{fig:loop}
\end{figure}
Now for any $L>0$ we can define a unique composite loop of length $L$ at $q_i^x$ as follows. Let $m$ be the largest integer satisfying $m \length(C_i^x)\leq L$. Then go $m$ laps around $C_i^x$, in the direction induced by $\ol T_x$, and one lap around $C_{i,\ell}^x$ for $\ell= L-m\length(C_i^x)$, if $\ell>0$. 
Since $\tilde v>|T'|=|\ol T'_{x_0}|$, we may choose $R$ so small that $\tilde v>|\ol T_x'|$ for all $x\in B$. Then
$$
\ell_i^x\coloneqq \int_{I_i} \tilde v\,dt - \length \mleft(\ol T_x\rvert_{I_i}\mright)>0.
$$
Finally, let $\tilde T_x$ be the curve of speed $\tilde v$ tracing $T$ plus  loops of length $\ell_i^x$ at $q_i^x$. Then  
$\ol T_x(I_i)\subset \tilde T_x(I_i)\subset U_i^x$. Thus
$|\tilde T_x-\ol T_x|_0< R/2$, which yields that $\lvert \ave(\tilde T_x) - \ave(\ol T_x)\rvert < R/2$ as desired.
\end{proof}

\section*{Acknowledgment}
We thank Andrzej Świ\k{e}ch for suggesting the proof of Lemma~\ref{LM:Kalman}, and Micha Wasem for comments on an earlier draft of this work.

\bibliographystyle{amsplain}
\bibliography{references}
\end{document}